\documentclass[11pt]{amsart}
\usepackage{verbatim, amscd, epsfig,amsmath,amsthm,amstext,amssymb,hyperref}
\usepackage[all]{xy}

\theoremstyle{plain}
\newtheorem*{theorem*}{Theorem}
\newtheorem{theorem}{Theorem}[section]
\newtheorem{lemma}[theorem]{Lemma}

\theoremstyle{definition}
\newtheorem{definition}[theorem]{Definition}

\newcommand{\R}{\mathbb{ R}}
\newcommand{\C}{\mathbb{ C}}
\newcommand{\Z}{\mathbb{ Z}}

\renewcommand{\P}{\mathbb{ P}}

\DeclareMathOperator{\Gr}{Gr}
\DeclareMathOperator{\Res}{Res}

\begin{document}
\title
[Spectral data for CMC tori]
{The Closure of Spectral Data for Constant Mean Curvature Tori in $ S ^ 3 $}

\author{Emma Carberry}
\email{emma.carberry@sydney.edu.au}
\address{Emma Carberry\\School of Mathematics and Statistics F07\\University of Sydney\\NSW 2006, Australia}
\author{Martin Ulrich Schmidt}
\email{schmidt@math.uni-mannheim.de}
\address{Martin Ulrich Schmidt\\Lehrstuhl f\"ur Mathematik III\\
Universit\"at Mannheim\\
D-68131 Mannheim, Germany}

\date{\today}

\maketitle

\begin {abstract}
The spectral curve correspondence for finite-type solutions of the sinh-Gordon equation describes how they arise from and give rise to  hyperelliptic curves with a real structure. Constant mean curvature (CMC) 2-tori in $ S ^ 3 $ result when these spectral curves  satisfy periodicity conditions. We prove that the spectral curves of CMC tori are dense in the space of smooth spectral curves of finite-type solutions of the sinh-Gordon equation. One consequence of this is the existence of countably many real $ n $-dimensional  families of CMC tori in $ S ^ 3 $ for each positive $ n $.
 \end {abstract}

\section {Introduction}
The study of 
constant mean curvature (CMC) 
surfaces has received considerable impetus over the last few decades with the realisation that 
they can be described by a loop of flat connections.
 In particular, the explicit bijective correspondence 
 \cite{PS:89, Hitchin:90, Bobenko:91} between conformally immersed CMC tori in 3-dimensional space forms and spectral curve data makes it possible to apply complex-algebraic techniques to the theory of these differential-geometric objects.



This algebraic viewpoint is clearly advantageous when considering  moduli space questions. Indeed spectral curves have been utilised to show that CMC tori in $\R ^ 3 $ , minimal tori in $ S ^ 3 $ and minimal Lagrangian tori in $\C P ^2 $ come in families of every real dimension (except dimension one in the first example) \cite { EKT:93, Jaggy:94,Carberry:04, CM:03}.
In each of these cases the  existence of a family of dimension $ g $ corresponds to the existence of a spectral curve of genus $ g $ and the impediment to the existence of such curves is periodicity conditions which the authors show can be satisfied for every $ g $. The algebro-geometric result proven by these authors is that for each $ g $ there is a curve $ Y $ of arithmetic genus $ g $ but geometric genus $ 0 $ such that in a neighbourhood of $ Y $, spectral curves of tori are dense in the space of spectral curves of finite type immersions of the complex plane. For CMC tori in $ S ^ 3 $, we show that for every $ g $, the density of curves satisfying periodicity conditions holds {\em globally}, not merely in the neighbourhood of a rational curve. 


The spectral curve viewpoint has proven fruitful in the study of CMC tori as by encoding geometric information in the algebraic data one obtains a new toolbox with which to prove theorems.
For example, Hitchin expressed the energy of a harmonic map into $ S ^ 3 $ explicitly in terms of spectral data  \cite 
 {Hitchin:90}, Ferus, Leschke, Pedit and Pinkall gave a lower bound for the area of CMC tori in $\R ^ 3 $ and Haskins bounded the 
geometric complexity of a special Lagrangian $ T ^ 2 $-cone \cite{Haskins:04} in terms of the genus of the respective spectral curves.
Kilian, the second author and Schmitt showed \cite{KSS:10} that given a spectral curve of a CMC torus in $ S ^ 3 $ of genus $ g $, it admits a one-parameter family of deformations. They gave a detailed analysis of this family of deformations in the case of equivariant CMC tori (which have spectral genus 0 or 1) and used this viewpoint to classify the equivariant minimal, embedded and Alexandrov embedded CMC tori. These results yield a partial version of Lawson's conjecture, namely that the Clifford torus is the only embedded minimal equivariant torus in $ S ^ 3 $. The preprint \cite {KS:08} by Kilian and the second author employs a related approach to prove the Pinkall-Sterling conjecture that the only embedded CMC tori in $ S ^  3 $ are surfaces of revolution and hence in particular the Lawson conjecture that the only embedded minimal torus in $ S ^ 3 $ is the Clifford torus.


CMC tori in $ S ^ 3 $ which are flat necessarily have spectral genus zero and such examples are known explicitly, each value of the mean curvature $ H $ yielding a unique embedded example up to isogeny.
The complete classification of all one-dimensional families of spectral genus one tori given in  \cite [Theorem 4.5] {KSS:10} a fortiori  demonstrates the existence of CMC tori of spectral genus one. For higher spectral genus the existence problem is solved by our density result. As a corollary of this result we conclude that for every $ n >0 $, there exist countably many real families of CMC tori in $ S ^ 3 $ of dimension $ n $.


A CMC surface admits a conformal parameterisation by curvature line coordinates in any neighbourhood of a non--umbilic point, and writing the conformal factor as $4e ^ u $ \cite[(14.3)]{Bobenko:91}, the function $ u $ satisfies the elliptic sinh-Gordon equation \cite[(14.6)]{Bobenko:91} \begin {equation}\label {eq:sG} u_{z\bar z} +2(1+H^2)e^u-\tfrac{1}{2}e^{-u} = 0, \end {equation} where $ z$ denotes a local complex curvature line coordinate. Conversely,  a solution of  \eqref {eq:sG} in a simply-connected domain determines a conformal CMC immersion, up to isometry. Thus on a simply-connected domain, a CMC immersion into $ S ^ 3 $ without umbilic points is equivalent to a solution of the sinh-Gordon equation. Umbilic points correspond to zeroes of the Hopf differential, which is the holomorphic quadratic differential given by the $ (2, 0) $ part of the second fundamental form. Hence aside from a totally umbilic surface (which must be a sphere), umbilic points are isolated and a CMC immersion of the plane which is periodic with respect to some lattice has no umbilic points. It is not true however that doubly periodic solutions of the sinh-Gordon equation must yield CMC immersions of tori, the corresponding surface could instead have translational symmetry.

The spectral data consists of a hyperelliptic curve $ X $ together with a meromorphic function $\lambda $ and line bundle $ L_0 $ on $ X $ and a pair of Sym points $\lambda_1\neq \lambda_2\in S ^ 1$. The periodicity of the CMC immersion with respect to some lattice $\Lambda\subset\C$ corresponds to the spectral curve satisfying periodicity conditions. If these conditions are satisfied then any line bundle satisfying reality conditions can be chosen and yields a unique CMC torus; the available space of line bundles has real dimension $ g $ and hence together with the 1-parameter family of deformations found in \cite{KSS:10}, for a spectral curve of genus $ g $, the corresponding CMC tori come in families of dimension $ g +1 $. The periodicity requirements can be geometrically described in terms of the singular curve $\tilde X$ formed from $ X $ by identifying together each of the pairs of points defined by the Sym points $\lambda_1,\lambda_2 $. The periodicity conditions are that the tangent plane to the Abel-Jacobi image of $\tilde X $ with base point $\lambda ^{-1}(0) $ is a rational plane in  $ \mathrm{Jac}(\tilde X) $. We may view this periodicity as coming in two parts: the requirement that the tangent plane to the Abel-Jacobi image of $ X $ be rational and the additional condition that this is true also for $\tilde X $. The first of these corresponds to the solution to the sinh-Gordon equation being doubly periodic, the second says additionally that the corresponding CMC surface is a torus.

The mathematical statement that we prove then is that amongst smooth hyperelliptic curves of any genus $ g $ satisfying the appropriate reality conditions, those which also satisfy both types of periodicity conditions form a dense subset. This is the content of Theorem~\ref {theorem:sphere}. As explained above the geometric interpretation of this is that in the space of (spectral curves of) CMC immersions of the plane into $ S ^ 3 $ of finite-type, the doubly periodic immersions form a dense subset.

We now briefly describe how this result is obtained. Periodicity corresponds to the tangent plane to the Abel-Jacobi image of the spectral curve being rational and rational planes are clearly dense in the appropriate Grassmanian. It hence suffices to show that the differential of the map which associates this plane to the spectral curve is generically invertible. To accomplish this, we consider polynomials $ b_1 $ and $ b_2 $ which geometrically correspond to choosing linearly independent vectors in the plane described above. In Lemma~\ref {lemma:isomorphism} we show that if this plane satisfies a certain pair of restrictions then the desired differential is invertible, as required. In the remainder of the manuscript we prove that for each genus $ g $ the set $\mathcal R ^ g $ of spectral curves satisfying these two restrictions is dense in the space $\mathcal H ^ g $ of all spectral curves of finite-type solutions to the sinh-Gordon equation.
We consider the pair of conditions separately, the space $\mathcal H ^ g\setminus\mathcal S ^ g $ of spectral curves satisfying just that the polynomials $ b_1 $ and $ b_2 $ have no common roots is the complement of an analytic subvariety and hence is either empty or dense in the space of finite-type spectral curves. 
We show by induction that for each $ g $ this space is indeed non-empty. Given an algebraic curve of genus $ g $ satisfying the appropriate reality conditions and with an ordinary double point on the unit circle,  in Lemma~\ref {lemma:polynomialperturbation} we ``open up" the singularity by adding a handle, obtaining a family of spectral curves of genus $ g +1 $.
The  meromorphic function $ f = b_1/b_2 $ on the spectral  curve is intrinsically defined up to M\"obius transformations and as shown in Lemma~\ref {lemma:induction} this process increases the degree of $ f = b_1/b_2 $ by one, which provides the necessary induction step for the behaviour of the common roots. We then use this in Lemma~\ref {lemma:dense} to prove that  the set $\mathcal R ^ g $ of spectral curve satisfying both conditions is also either empty or dense in $\mathcal H ^ g $. In  Lemma~\ref {lemma:induction 2} we calculate how the second condition defining $\mathcal R ^ g $ behaves under the process of adding a handle and we employ this in Lemma~\ref{lemma:non empty}  to show that for each genus $ g $, the space $\mathcal R ^ g $ is  non-empty.

In section 2, we explain how CMC tori in $ S ^ 3 $ are characterised in terms of spectral data and in section 3 we prove our main density result, Theorem~\ref {theorem:sphere}.

We expect similar statements to hold for spectral data of tori with other geometric properties. For example in the case of CMC tori in Euclidean three-space, the periodicity conditions force the polynomials $b_1$ and $b_2$ to have a common root at the Sym point on $S^1$. 
We expect that the closure of smooth spectral curves of CMC tori in Euclidean three-space is the subset of all smooth hyperelliptic curves of any genus $ g $ satisfying the appropriate reality conditions such that $b_1$ and $b_2$ have one common root on $S^1$.

\section {Spectral Curve Approach}
We begin by explaining the integrable systems approach to CMC tori in
the 3-sphere and the contruction of such tori with the
help of spectral curves.

For any conformal map from $\R^2$ to
$S^3\subset\R^4$ the position vector together with the normalised
partial derivatives and the normal vector to the tangent plane in
$S^3$ form an orthonormal oriented basis of $\R^4$, i.e. a
map from $\R^2$ to $SO(4,\R)$, which is called a frame. This frame can be
lifted to a map to $SU(2,\C)\times SU(2,\C)$, which is a two-sheeted
covering of $SO(4,\R)$. The Maurer-Cartan
form of this frame can be expressed in terms of the first and second
fundamental forms. For conformal constant mean curvature immersions
the Maurer-Cartan equation is equaivalent to the $\sinh$-Gordon
equation. In \cite{Bobenko:91} the integrable structure of this
equation was used to construct a family of $su(2,\C)$-valued
solutions of the Maurer-Cartan equation parameterized by
$\lambda\in S^1$, such that the Maurer-Cartan form of the frame is the
Cartesian product of this family evaluated at two values
$\lambda_1\neq\lambda_2\in S^1$. Furthermore, this family extends
holomorphically to a $sl(2,\C)$-valued Maurer-Cartan form depending on
$\lambda\in\C^{\times}$. If $F_{\lambda}$ is a
$\lambda$-dependent family of maps to $SL(2,\C)$, whose derivative is
this family of Maurer-Cartan forms, then the immersion is up to an
isometry of $S^3$ equal to $\psi=F_{\lambda_1}^{-1}F_{\lambda_2}$
(\cite[(14.11)]{Bobenko:91}.

In \cite{PS:89, Hitchin:90} it was shown that all doubly periodic
solutions on $\R^2$ of the $\sinh$-Gordon equation are of
finite type. In particular all such solutions give rise to a
hyperelliptic curve called the spectral curve and a holomorphic
line bundle $L_0$ on this curve. In \cite{Bobenko:91} the theory of
integrable systems was applied to the $\sinh$-Gordon equation for the case
of non-singular spectral curves of finite genus. In this case the corresponding
Baker-Akhiezer functions and the corresponding solutions of the
$\sinh$-Gordon equation can be expressed in terms of theta
functions of the non-singular spectral curve. Since the Baker-Akhiezer
function defines an integral $F_\lambda$ of the $\lambda$-dependent
Maurer-Cartan form, the immersion can be expressed in terms of
theta functions of the spectral curve
\cite[(16.4)]{Bobenko:91}. Furthermore, these immersions are
doubly periodic if and only if the spectral curves obey certain
periodicity conditions \cite[(16.5)]{Bobenko:91}. These conditions
extend to singular spectral curves.

From the information of the spectral curves, the map $\lambda $, Sym
points $\lambda_1,\lambda_2 $ and a line bundle $ L_0 $, the CMC
immersion $ \psi $ can be recovered. The spectral curve satisfies
reality conditions which are detailed in the next section, along with
the periodicity conditions. Given an appropriate curve, the line
bundle can be chosen from a real $ g $-dimensional family.

\section{Spectral curves of constant mean curvature tori in $ S ^ 3 $}
We shall prove that the set of spectral curves of constant mean curvature (CMC) 2-tori in $ S ^ 3 $ is dense in the space of smooth spectral curves of finite-type solutions of the sinh-Gordon equation.

The spectral curves of finite type solutions of the sinh-Gordon equation are hyperelliptic curves $ X_a $ given by 
 $ y ^ 2 =\lambda a (\lambda) $ 
where $ a $ is a polynomial of degree $ 2g $ satisfying the reality condition
\begin {equation}
\lambda ^ {2 g}\overline {a(\bar\lambda^{-1})}=a(\lambda)\label{eq:reality}
\end{equation}
and the positivity condition
\begin {equation}
\lambda ^ {-g} a (\lambda) >0\quad\text {for all}\quad\lambda\in S ^ 1. \label {eq:positivity}
\end {equation}
We denote by $\mathcal H ^ g $ the set of all degree $ 2g $ polynomials $ a $  satisfying  \eqref {eq:reality} and \eqref {eq:positivity} whose roots are pairwise distinct and whose highest coefficient has absolute value one. The condition on the roots ensures that the resulting hyperelliptic curve is smooth and irreducible, whilst the last condition is a normalisation. These polynomials are uniquely determined by their roots. Writing
\[
a (\lambda) =
(-1) ^ g\prod_{j = 1} ^ g\frac {\bar\eta_j} {\left|\eta_j\right|} (\lambda -\eta_j) (\lambda -\bar\eta_j^ {- 1}),
\]
each $ a\in
\mathcal H ^ {g} $ corresponds to $ g $ pairwise distinct elements $\{\eta_1,\ldots,\eta_g\} $ of $\{z\in\C\mid 0 <\left| z\right| <1\}$, \label {Hg}
and we take on $\mathcal H ^ g $ the metric induced by this by bijection.

For each $ a\in \mathcal{H} ^ g $, define an antiholomorphic involution
\begin {align*}
 \rho: X_a &\rightarrow X_a,&
(\lambda, y) &\mapsto \Bigl(\bar\lambda^{-1}, \;\frac {\bar y} {\bar\lambda ^ {(g +1)}}\Bigr),
\end {align*}
and note that $\rho $ fixes the points on $ X_a $ with $\lambda\in S ^ 1 $.

Each hyperelliptic curve $ X_a $ with $ a\in\mathcal H^ g $ is the spectral curve of a real $ g $-dimensional family of finite-type solutions to the sinh-Gordon equation, where these solutions are parameterised by line bundles of degree $ g +1 $ on $X_a$ which are quaternionic with respect to the antiholomorphic involution $\sigma\rho $.
The corresponding solutions of the sinh-Gordon equation are doubly periodic if and only if there exist functions $\mu_1 $ and $\mu_2 $ on the spectral curve which transform under the hyperelliptic involution $\sigma $ according to $\sigma ^*\mu_1 =\frac {1} {\mu_1},\;\sigma ^*\mu_2 =\frac {1} {\mu_2} $ and have differentials of the form
\begin{align}\label{eq:diff}
dq_1=d\log\mu_1&=\frac {b_1 (\lambda)} {y}\frac {d\lambda} {\lambda},&
dq_2=d\log\mu_2&=\frac {b_2 (\lambda)} {y}\frac {d\lambda} {\lambda},
\end{align}
where $ b_1, b_2 $ are $\R $-linearly independent polynomials of degree $ g +1 $, satisfying
\begin {equation}
\lambda ^ {g +1}\overline {b_1 (\bar\lambda^{-1})} = b_1 (\lambda),\quad \lambda ^ {g +1}\overline {b_2 (\bar\lambda^{-1})} = b_2 (\lambda).\label {eq:breality}
\end{equation}
For each $a\in\mathcal H ^ g $, let $ \mathcal B_a $ be the space of polynomials $ b $ of degree $ g +1 $ satisfying the reality condition $\lambda ^ {g +1}\overline {b(\bar\lambda^{-1})} = b (\lambda) $ and such that the differential form $\Theta_b:=\frac {b (\lambda)} {y}\frac {d\lambda} {\lambda} $ has purely imaginary periods. The reality condition on $ b $ is equivalent to
\begin{equation}
\overline {\rho ^*(\Theta_b)} =-\Theta_b.\label{eq:Thetareality}
\end{equation}

We may choose a canonical basis $ A_1,\ldots, A_g, B_1,\ldots, B_g $ for the homology of $ X_a $ such that $\rho_*(A_j) \equiv - A_j $,  $\rho_*(B_j) \equiv B_j\;\mathrm{mod}\; \langle A_1, \ldots , A_g \rangle $ and the  projection of each $ A_j $ to $\C\P ^ 1 $ has winding number one around  $\eta_j,{\bar\eta_j}^ {- 1} $ and winding number zero around the other roots of $ a $. We then call $ A_1, \ldots , A_g$, $B_1,\ldots, B_g $ an {\it adapted canonical basis}. The condition on periods clearly imposes $ g $ linearly independent constraints $\int_{A_j}\frac {b (\lambda)} {y}\frac {d\lambda} {\lambda} = 0 $ on the $ (g +2) $-dimensional space of polynomials of degree $ g +1 $ satisfying the reality condition, so we see that $ \mathcal B_a $ has real dimension two. When the spectral curve corresponds to a doubly periodic solution of the sinh-Gordon equation, $ \mathcal B_a $ is  spanned by the polynomials $ b_1, b_2 $ above, whose periods lie in $ 2\pi i\Z $.

A spectral curve of a doubly periodic solution of the sinh-Gordon equation is the spectral curve of a CMC torus in $ S ^ 3 $ if and only if there exist distinct $\lambda_1,\lambda_2\in S ^ 1 $ such that
\[
\mu_1 (\lambda_1) =\mu_1 (\lambda_2) =\pm 1\quad\mbox{and}\quad \mu_2 (\lambda_1) =\mu_2 (\lambda_2) =\pm 1.
\]
For each $\lambda_1\neq\lambda_2\in S ^ 1 $, denote by $\mathcal H ^ g (\lambda_1,\lambda_2) $ the space of all $ a\in\mathcal H ^ g $ such that $ (a,\lambda_1,\lambda_2) $ are the spectral data of a family of CMC tori in $ S ^ 3 $.

\begin {theorem} \label{theorem:sphere}
For each $\lambda_1\neq\lambda_2\in S ^ 1 $, the space $\mathcal H ^ g (\lambda_1,\lambda_2)$ is a dense subspace of $\mathcal H ^ g $.
\end {theorem}
\begin {proof}

Take an adapted canonical basis $ A_1,\ldots, A_g, B_1,\ldots, B_g $ for the homology of $ X_a $.
For $a\in\mathcal H ^ g $ and $\lambda_1\neq\lambda_2\in S ^ 1 $, choose paths $ \gamma_1, \gamma_2 $ in $ X_a $ such that for each $ j = 1, 2 $, the path $ \gamma_j $ connects the two points of $ X_a $ over $\lambda_j $ and $ \rho_*(\gamma_j) \equiv\gamma_j\;\mathrm{mod}\; \langle A_1, \ldots , A_g \rangle $. Then we consider
\begin {align*}
\varphi ^ {(a,\lambda_1,\lambda_2)}:  \mathcal B_a &\rightarrow \R ^ {g +2},&
 b &\mapsto \frac {1} {2\pi i}\left (
\int_{B_1}\Theta_b,\;\;\ldots,\;\;\int_{B_g}\Theta_b,\;
\int_{\gamma_1}\Theta_b,\;\int_{\gamma_2}\Theta_b \right).
\end {align*}
The above periodicity conditions tell us that $ (a,\lambda_1,\lambda_2) $ corresponds to a CMC 
 torus in $ S ^ 3 $ if and only if there exist linearly independent $ b_1, b_2 $ in the real 2-dimensional vector space $ \mathcal B_a $ such that
\[
\varphi ^ {(a,\lambda_1,\lambda_2)} (b_1)\in\Z ^ {g +2}\quad\mbox{and}\quad\varphi ^ {(a,\lambda_1,\lambda_2)} (b_2)\in\Z ^ {g +2}.
\]
This is equivalent to the condition that the image of $\varphi ^ {(a,\lambda_1,\lambda_2)} $ is a rational 2-plane in $\R ^ {g +2} $, and this condition is independent of the choice of $\gamma_1,\gamma_2 $.
Let $\Gr ^ {g +2}_2 (\R)$ denote the Grassmanian of 2-planes in $\R ^ {g +2} $. 
\begin {align*}
\Phi ^ {(\lambda_1,\lambda_2)}:\mathcal H ^ g &\rightarrow\Gr ^ {g +2}_2(\R),&
a &\mapsto \varphi ^ {(a,\lambda_1,\lambda_2)} (\mathcal B_a)
\end {align*}
is a smooth map between $2g$-dimensional manifolds.

In Definition~\ref{definition:R} we introduce a subset
$\mathcal R^g\subset\mathcal H^g$ such that
\begin {enumerate}
\item[(I)] $\mathcal R ^ g $  is due to Lemmata~\ref{lemma:dense} and \ref{lemma:non empty} an open and dense subset of $\mathcal H ^ g $.
\item[(II)] For each $ a\in\mathcal R ^ g $ and $\lambda_1\neq\lambda_2\in S ^ 1 $,  the differential $ d\Phi ^ {(\lambda_1,\lambda_2)}_a $ is due to Lemma~\ref {lemma:isomorphism} an isomorphism.
\end {enumerate}
Since $ (a,\lambda_1,\lambda_2) $ corresponds to a CMC torus in $ S ^3 $ if and only if $\Phi ^ {(\lambda_1,\lambda_2)} (a) $ is a rational element of $\Gr ^ {g +2}_2(\R)$, these two statements together with the Inverse Function Theorem prove Theorem~\ref {theorem:sphere}.\end{proof}

In order to prove (I) we first introduce a subvariety $\mathcal S ^ g $ with $\mathcal R ^ g\subset\mathcal H ^g\setminus\mathcal S ^ g$ and show in Lemmata~\ref {lemma:polynomialperturbation} and \ref{lemma:induction} that $\mathcal H ^ g\setminus\mathcal S ^ g $ is open and dense in $\mathcal H ^ g $.
\begin {definition}
Let $ \mathcal  S ^ g $ denote the set of those $ a\in\mathcal H ^ g $ for which linearly independent $ b_1, b_2 \in\mathcal B_a $ have a common root in $\C ^\times $.\end {definition}
To see that this is well defined, independent of the choice of $ b_1$ and $ b_2 $, consider for each $ a\in\mathcal H ^ g $ the meromorphic function on $\C \P ^ 1 $ defined by the quotient $ f (\lambda) =\frac {b_1 (\lambda)} {b_2(\lambda)}$ of two linearly independent elements $ b_1, b_2 $ of $ B_a $. This function is determined up to M\"obius transformations by $ a\in\mathcal H ^ g $ and hence its branch locus is independent of the choice of $ b_1, b_2\in B_a $. The reality condition \eqref{eq:breality} implies that $ f $ maps $\lambda\in S ^ 1 $ into the one-point compactification $\R\P ^ 1 $ of $\R $ in $\C \P ^ 1 $. Note that $ f:\C\P ^ 1\rightarrow\C\P ^ 1 $ has degree $ g +1 $ if and only if $ b_1 $ and $ b_2 $ have no common zeroes, and that the degree of $ f $ is unchanged by M\"obius transformations, so that if there exist linearly independent $ b_1, b_2\in\mathcal H ^ g $ having no common zeros, then this property holds for any pair of linearly independent elements of $\mathcal H ^ g $. 

Denoting the zeros of $ b_k $ by $\beta ^ 1_k,\ldots,\beta ^ {g +1}_k $, we define the {\em discriminant} of $ f $ to be
\[
\Delta=\prod_{j_1 = 1} ^ {g +1}\prod_{j_2 = 1} ^ {g +1} (\beta ^ { j_1}_1 -\beta ^ {j_2}_2).
\]
Then $ \mathcal  S ^ g $ is equivalently the set of those $ a\in\mathcal H ^ g $ for which $\Delta = 0 $.
For a given polynomial $ b $ of degree $ g +1 $ and a local choice of canonical basis $ A_1,\ldots, A_g, B_1,\ldots, B_g $ of $ H_1 (X_a,\Z) $ for all $ a $ in an open neighbourhood of $ a_0\in\mathcal H ^ g $, the periods
\[
\int_{A_1}\Theta_b,\cdots,\int_{A_g}\Theta_b,
\int_{B_1}\Theta_b,\cdots,\int_{B_g}\Theta_b
\]
depend holomorphically on $ a $ in this neighbourhood. This implies that the set $\mathcal S ^ g $ 
 is a real subvariety of $\mathcal H ^ g $.
As $\mathcal H ^ g $ is connected, the subvariety $\mathcal S ^ g $ is either all of $\mathcal H ^ g $, or is the complement of an open and dense subset.

Take $ a\in\mathcal H ^ g\setminus\mathcal S ^ g $ and let $ b_1, b_2\in\mathcal B_a $ be linearly independent. Then there are polynomials $ c_1, c_2 $ of degree at most $ g $ such that
\begin {equation}\label {eq:R}
c_1 b_2 - c_2 b_1 = a\text { or equivalently }\frac {c_1} {b_1} -\frac {c_2} {b_2} =\frac {a} {b_1 b_2},
\end {equation}
and since $ a\notin\mathcal S ^ g $, these polynomials are unique. Note that the reality conditions \eqref {eq:reality} on $ a $ and \eqref {eq:Thetareality} on $ b_1, b_2 $ force $ c_1, c_2 $  to satisfy
\begin {align*}
\lambda ^ {g-1}\overline {c_1 (\bar\lambda^ {- 1})} &= c_1(\lambda),& \lambda ^ {g-1}\overline {c_2 (\bar\lambda^ {- 1})} &= c_2 (\lambda).
\end{align*}
Since $ a $ has degree $ 2g $ and $ b_1b_2 $ has degree $ 2g +2 $, the meromorphic form $\frac {a} {b_1 b_2} d\lambda $ has no residue at $\lambda =\infty $. Hence the following statements are equivalent:
\begin {enumerate}
\item degree$ (c_1)=g $,
\item degree$ (c_2)=g $,
\item $\sum_{\text {roots of } b_1}\Res\,\frac {a} {b_1b_2} d\lambda \not= 0 $,
\item $\sum_{\text {roots of } b_2}\Res\,\frac {a} {b_1b_2} d\lambda \not= 0 $.
\end {enumerate}
\begin {definition}\label {definition:R}
Define $\mathcal R ^ g $ to be the set of $ a\in\mathcal H ^ g\setminus\mathcal S ^ g $ such that the (equivalent) conditions above are satisfied.
\end {definition}
We proceed now to show that this definition is independent of the choice of basis $ b_1, b_2 $ of $\mathcal B_a $. Any other such basis is given by
\[
\tilde b_1 = Ab_1+ Bb_2,\,\tilde b_2 = Cb_1+ Db_2\text { with }\left (
\begin {array} {cc}
A & B\\
C & D\end {array}\right)\in GL (2,\R).
\]
Then defining
\[
\tilde c_1 =\frac {Ac_1+ Bc_2} {AD - BC},\,\tilde c_2 =\frac {Cc_1+ Dc_2} {AD - BC}
\]
we have
\[
\tilde c_1\tilde b_2 -\tilde c_2\tilde b_1 = c_1b_2 - c_2b_1 = a.
\]
Then clearly  $ c_1$ and $ c_2 $ have degree $ g $ if and only if the same is true of $ \tilde c_1 $ and $\tilde c_2 $.
In fact in the sequel it will be useful to have the following slightly stronger statement.
\begin {lemma}\label {lemma:independent} Take $ a\in\mathcal H ^ g\setminus\mathcal S ^ g $.
For any change of basis
\[
\tilde b_1 = Ab_1+ Bb_2,\;\tilde b_2 = Cb_1+ Db_2, \qquad\left (
\begin {array} {cc}
A & B\\
C & D\end {array}\right)\in GL (2,\R)
\]
of $\mathcal B_a $ we have
\[
\sum_{\text {\rm roots of } \tilde b_1}\Res\,\frac {a} {\tilde b_1\tilde b_2} d\lambda 
= \frac {1} {AD - BC}\sum_{\text {\rm roots of } b_1}\Res\,\frac {a} {b_1b_2} d\lambda.
\]
In particular then this quantity is invariant under M\"obius transformations. 
\end {lemma}
\begin {proof} This is a straightforward computation, namely
\begin {align*}
\sum_{\text {roots of } \tilde b_1}\Res\,\frac {a} {\tilde
  b_1\tilde b_2} d\lambda 
& = \sum_{\text {roots of } \tilde b_1}\Res\, \frac {\tilde c_1} {\tilde b_1} d\lambda \\
& = -\underset{\lambda =\infty }{\Res}\,\frac {\tilde c_1} {\tilde b_1}d\lambda\text { (Residue Theorem)}\\
& = \frac {-1} {AD - BC}\,\,\underset{\lambda =\infty }{\Res}\,\frac {c_1} { b_1}d\lambda\text { (direct computation)}\\
& = \frac{1} {AD - BC} \sum_{\text {roots of }  b_1}\Res\, \frac {c_1} {b_1} d\lambda \\
& = \frac {1} {AD - BC}\sum_{\text {roots of } b_1}\Res\,\frac {a} {b_1b_2} d\lambda.
\end {align*}
\end {proof}

\begin {lemma}\label{lemma:isomorphism}
For all $ a\in\mathcal R ^ g $ and $\lambda_1\neq\lambda_2\in S ^ 1 $, the differential $ d\Phi ^ {(\lambda_1,\lambda_2)}_a $ is an isomorphism.
\end {lemma} 
\begin {proof}
Write $ P ^ {g +1} $ for the space of polynomials of degree $ g+1 $. On a neighbourhood $ U $ of $a\in \mathcal H ^ g $, we may choose a smoothly varying adapted canonical basis for $ H_1 (X_a,\Z) $ and curves $\gamma_1,\gamma_2 $  as in the proof of Theorem~\ref{theorem:sphere} and a smoothly varying basis $ b_1 (a), b_2 (a) $ of $ \mathcal B_a$. These choices define a smooth map
\begin {align*}
\Psi :  U &\rightarrow P ^ {g +1}\times P ^ {g +1},&
a &\mapsto (b_1 (a),\, b_2 (a))
\end {align*}
 such that $\Phi ^ {(\lambda_1,\lambda_2)} $ factors as
\begin{align*}
U \xrightarrow{\Psi}  &\R ^ {g +2}\times\R ^ {g +2} \rightarrow \Gr ^ {g +2}_2(\R),&
a \mapsto &\left(\varphi _{a} (b_1 (a)),\varphi _{a} (b_2 (a))\right) \mapsto\Phi ^ {(\lambda_1,\lambda_2)} (a).
\end{align*}

Suppose $ a_t\in\mathcal H ^ g $ is for $ t\in (-\delta,\delta) $ a smooth family such that $\dot a =\left.\dfrac {d} {dt}a_t\right|_{t =   0} \in T_a\mathcal H ^ g $ is in the kernel of $ d\Phi ^{(\lambda_1,\lambda_2)}_a $. Then we may choose the functions $ b_{1t}=b_1(a_t)$ and $b_{2t}=b_2(a_t)$ so that $ d\Psi_a (\dot a) = 0 $ also. In the following calculation we consider $\lambda$ and $t$ as independent variables. Consequently the polynomials $a_t$, $b_{1t}$ and $b_{2t}$ define functions depending on $\lambda$ and $t$ whilst the solution $y$ of $y^2=\lambda a_t(\lambda)$ may be considered as a two-valued function depending on $\lambda$ and $t$. For such  functions $ f $, $ f' $ denotes the derivative with respect to $\lambda $, whilst $\dot f $ denotes the derivative with respect to $ t $, evaluated at $ t = 0 $. Recalling the multi-valued functions $ q_1, q_2 $ defined in \eqref{eq:diff}, the derivatives $\frac{dq_1}{d\lambda}$ and $\frac{dq_2}{d\lambda}$ are (single-valued) meromorphic functions of $\lambda $ and $ t $ and  $\left.\frac {d} {dt}\right|_{t = 0}\frac{d q_1}{d\lambda}d\lambda$ and $\left.\frac {d} {dt}\right|_{t = 0}\frac{dq_2}{d\lambda}d\lambda $ are meromorphic differentials without residues, whose periods all vanish. Such differentials must be the exterior derivative of a meromorphic function. Hence there exist meromorphic functions $\dot q_1 $ and $\dot q_2 $, such that
\begin {equation}\label{eq:derivativeTheta}
d\dot{q_1}=\left.\frac{d}{dt}\right|_{t=0} \Theta_{b_{1t}}\quad\mbox{and}\quad d\dot{q_2}=\left.\frac{d}{dt}\right|_{t=0}\Theta_{b_{2t}}.
\end {equation}
Then also $\dot q_1,\dot q_2 $ satisfy the reality condition
  $\rho_*(\dot q_j) =-\overline {\dot q_j} $, so there exist
  polynomials $ \hat c_1 $ and  $ \hat c_2 $ of degree $ g +1 $
  such that
\begin {align*}
\dot q_1 &=\frac {i\hat c_1 (\lambda)} {y},&\dot q_2 &=\frac {i\hat c_2 (\lambda)} {y},&\lambda ^ {g+1}\overline {\hat c_1 (\bar\lambda^{- 1})} &=\hat c_1(\lambda),& \lambda ^ {g+1}\overline {\hat c_2 (\bar\lambda^ {- 1})} &=\hat c_2 (\lambda).
\end {align*}
From  \eqref {eq:derivativeTheta} we have
\[
\dfrac {\partial} {\partial\lambda}\frac {i\hat c_1 (\lambda)} {y} =\left.\dfrac {\partial} {\partial t}\frac {b_1 (\lambda)} {y\lambda}\right |_{t = 0}\quad\text {  and }\quad\dfrac {\partial} {\partial\lambda}\frac {i\hat c_2 (\lambda)} {y} =\left.\dfrac {\partial} {\partial t}\frac {b_2 (\lambda)} {y\lambda}\right |_{t = 0},
\]
which imply
\begin {equation} 2\lambda a i\hat c_1' - a i\hat c_1 -\lambda a' i\hat c_1  = 2 a\dot b_1 -\dot ab_1
\label {eq:derivative1}
\end {equation} and
\begin {equation}
2\lambda ai\hat c'_2 - ai\hat c_2 -\lambda a' i\hat c_2  = 2a\dot b_2 -\dot ab_2.\label {eq:derivative2}
\end {equation}
We have assumed that $ b_1 $ and $ b_2 $ have no common root. If $
\hat c_1 $ and $ \hat c_2 $ are both the zero function then the
equations \eqref{eq:derivative1}-\eqref{eq:derivative2} therefore imply that $\dot a $ vanishes at all roots of $ a $. But $ a $ has distinct roots and the highest coefficient has absolute value one. Then \eqref{eq:reality} and \eqref{eq:positivity} uniquely determine $a$ in
terms of its roots, so $\dot a\equiv  0 $. Thus to prove that $\dot a $ is trivial it suffices to show that $ \hat c_1, \hat c_2 $ both vanish identically.

Locally on simply connected subsets $ U_\beta $ of the spectral curve $ X_a $, there are meromorphic functions $ q_1 $ and $ q_2 $, whose exterior derivatives are equal to
\[
dq_1 =\frac {b_1 (\lambda)} {   y}\frac {d\lambda} {\lambda}\quad\text { and } \quad dq_2 =\frac {b_2 (\lambda)} {y}\frac {d\lambda} {\lambda}.
\]
The image of each $ U_\beta $ under the map $ (q_1, q_2) $ is the zero set of a holomorphic function $ (q_1, q_2)\mapsto R_\beta (q_1, q_2) $ on an open subset of $\C \P ^ 1\times\C\P ^ 1 $. Furthermore, every point $x\in X_a$ is contained in such an open neighbourhood $U_\beta\subset X_a$, such that $q_1:U_\beta\to V_\beta$ is a proper $l$-sheeted covering onto an open subset $V_\beta\subset \C \P ^ 1$. Here $l-1$ is the vanishing order of $dq_1$ at the point $x$. Due to \cite[Theorem~8.3]{Forster:81} there exists a unique choice of $R_\beta$ such that $q_2\mapsto R_\beta(q_1,q_2)$ is a polynomial of degree $l$ with highest coefficient one, whose coefficients are meromorphic functions depending on $q_1\in V_\beta$. For smooth families $ a_t\in\mathcal  H ^ g $, these functions $ R_\beta (q_1, q_2) $ depend smoothly on $ a_t \in\mathcal H ^ g $. Hence the $\frac {\partial R_\beta} {\partial t}  (q_1, q_2) $ are also holomorphic functions on open subsets of $\C \P ^ 1\times\C \P ^ 1 $. We have that
\[
0 =\frac {d} {dt} R_\beta (q_1, q_2) =\frac {\partial R} {\partial t} (q_1, q_2) +\frac {d  q_1} {dt}\frac {\partial R } {\partial q_1}(q_1, q_2) +\frac {d  q_2} {dt}\frac {\partial R } {\partial q_2}(q_1, q_2)
\]
and for fixed $ t $,
\[
0 =\frac {\partial R} {\partial q_1} (q_1, q_2) dq_1+\frac {\partial R} {\partial q_2} (q_1, q_2) dq_2.
\]
This gives that
\[
\dot q_1dq_2 -\dot q_2dq_1
=\frac {-\frac {\partial R} {\partial t} (q_1, q_2)} {\frac {\partial R} {\partial q_1} (q_1, q_2)} dq_2=\frac {\frac {\partial R} {\partial t} (q_1, q_2)} {\frac {\partial R} {\partial q_2} (q_1, q_2)} dq_1.
\]

The above equation shows that the meromorphic global differential form $\dot q_1dq_2 -\dot q_2dq_1 $ is holomorphic away from $ P_0 $ and $ P_\infty $. Moreover, it is invariant with respect to the hyperelliptic involution $\sigma(\lambda,y)=(\lambda,-y)$, as  $\sigma ^*(q_1) = - q_1 $ and $\sigma^*(q_2)=-q_2$. Since this form has at most third order poles at $ P_0 $ and $ P_\infty $, we conclude that it may be written as
\begin{equation}
\dot q_1dq_2 -\dot q_2dq_1 =\frac {Q (\lambda) d\lambda} {\lambda ^ 2},\label {eq:Q}
\end {equation}
with $ Q $ a polynomial of degree two. We have assumed that the $ t $-derivatives of $\int_{\gamma_1} dq_1,\int_{\gamma_2} dq_1,\int_{\gamma_1} dq_2,\int_{\gamma_2} dq_2 $ vanish when $t=0$. Thus $\dot q_1 $ and $\dot q_2 $ and $\hat c_1$ and $\hat c_2$ all vanish at $\lambda_1 $ and $\lambda_2 $, so that $ Q (\lambda) = iC (\lambda -\lambda_1) (\lambda -\lambda_2) $ for some $ C\in\C $. We may write $ \hat c_1 (\lambda) =(\lambda -\lambda_1) (\lambda-\lambda_2) c_1 (\lambda) $ and $ \hat c_2 (\lambda) = (\lambda-\lambda_1) (\lambda -\lambda_2) c_2 (\lambda) $ with polynomials $c_1$ and $c_2$ of degree $g-1$. Then as $ b_1 $ and $ b_2 $ do not have a common root, equation \eqref{eq:Q} is equivalent to
\begin {equation}
c_1b_2 - c_2b_1 = C a.\label {eq:c}
\end {equation}
By definition of $\mathcal R^g$ the unique solution to \eqref{eq:c} is that $C$, $ c_1 $ and $ c_2 $ vanish. As shown earlier, this implies $\dot a = 0 $, and so $ d\Phi ^ {(\lambda_1,\lambda_2)}_a $ has trivial kernel.\end{proof}

To  prove (I) and hence complete the proof of  Theorem~\ref{theorem:sphere} it remains
then to show that for each $ g\in\Z^+\cup\{0\} $, $\mathcal R^g$ is dense in $\mathcal H ^ {g }$.

We shall first prove that for each $g\in\Z^+\cup\{0\}$, the subset $\mathcal H^g\setminus\mathcal S^g$ is open and dense in $\mathcal H^g$. Since $\mathcal{S}^g$ is a subvariety, it suffices to prove that it is non-empty for each $g$, which we shall prove by induction. For the unique spectral curve $ y ^ 2 =\lambda $ of genus zero,  linear functions $ b_1, b_2\in\mathcal B_\lambda $ are linearly independent if and only if they have no common roots. Lemmata~\ref {lemma:polynomialperturbation} and \ref{lemma:induction} will yield the induction step.

\begin {lemma}\label {lemma:polynomialperturbation}
Fix $ a\in\mathcal H ^ g $, $ b\in  B_a $, $\alpha\in S ^ 1 $ and a choice of $\sqrt  {\bar\alpha} $. There exists $\epsilon >0 $ such that for each $ t\in (-\epsilon, 0)\cup (0,\epsilon) $, 
the polynomial
\[
a_t (\lambda) = (\lambda -\alpha e ^ { t}) (\bar\alpha\lambda - e ^ {-  t}) a (\lambda)
\]
lies in $ \mathcal H ^ {g +1} $ and for $ t\in (-\epsilon,\epsilon) $ the conditions 
\begin {enumerate}
\item     $ b_0 =\sqrt {\bar\alpha} (\lambda -\alpha) b $,
\item $ b_t (0) = -\alpha\sqrt {\bar\alpha}\, b (0) $
\end {enumerate}
determine a real-analytic family of polynomials $ b_t\in B_{a_t} $.

 Note that {\rm (1)} is equivalent to $\Theta_b =\iota_0 ^*(\Theta_{b_0}) $ where $\iota_0: X_a\rightarrow X_{a_0} $ denotes the normalisation map $ (\lambda, y)\mapsto (\lambda,\sqrt {\bar\alpha} (\lambda - a) y) $.

\end {lemma}
\begin {proof}

Let $ A_1,\ldots, A_g, B_1,\ldots, B_g $ be an adapted canonical basis of $ H_1 (X_a,\Z) $. For $\epsilon >0 $ sufficiently small, we may take $ A_{g +1}, B_{g +1}\in H_1 (X_{a_t},\Z) $ such that $ A_{g +1} $ coincides with the circle $\partial B_\epsilon (\alpha) $ and $$(\iota_t)_*(A_1),\ldots,(\iota_t)_*(A_g), A_{g +1},(\iota_t)_*(B_1),\ldots ,(\iota_t)_*(B_g), B_{g +1} $$ is an adapted canonical basis of $ H_1 (X_{a_t},\Z) $, where
\begin {align*}
\iota_t:  U\subset X_a &\rightarrow X_{a_t},&
(\lambda, y) &\mapsto (\lambda,\sqrt {(\lambda -\alpha e ^ t) (\bar\alpha\lambda - e ^ {- t}) } y.
\end {align*}
Here we have taken an open set $ U\subset X_a $ containing $ A_1, \ldots , A_g, B_1, \ldots , B_g $ such that $\iota_t |_U $ is a diffeomorphism onto its image.
The two real conditions $ b_t (0) =-\alpha\sqrt {\bar\alpha}\, b (0)$ determine a unique $ b_t\in B_{a_t} $.
We check that $ b_0 $ as defined above lies in $ B_{a_0} $. Indeed
\[
\Theta_{b_0} =
\frac {\sqrt {\bar\alpha} (\lambda -\alpha) b (\lambda)}
 {\sqrt {(\lambda -\alpha) (\bar\alpha\lambda -1)} y}\frac {d\lambda} {\lambda} =\frac {b (\lambda)} {y}\frac {d\lambda} {\lambda}
\]
has no pole at $\lambda =\alpha $, and hence $\int_{A_{g +1}}\Theta_{b_0} = 0 $.
Since $ a_t $ and the curve $ A_{g +1} ^ t $ are real-analytic in $ t \in (-\epsilon,\epsilon)$, so is the family $ b_t $, in the sense that the coefficients of the polynomials $ b_t $ are real-analytic functions on the open interval $ (-\epsilon,\epsilon) $. 
\end {proof}
Note that for $\alpha\not\in S ^ 1 $, the analogous argument does not hold for families with two pairs of additional roots nearby each of $\alpha $ and $\bar\alpha ^{-1} $, since then the additional $ A $-cycles are not well-defined.

For each linearly independent $ b_1, b_2\in B_a $ 
 let $ b_{1t}, b_{2t} $ be families of linearly independent functions in $ B_{a_t} $ defined as in Lemma \ref{lemma:polynomialperturbation}. 
We now consider $ f_t =\frac {b_{1t}} {b_{2t}} $, keeping in mind that the roots of $ df_t $ are independent of the choice of $ b_1, b_2 $ and depend real-analytically on $t\in (-\epsilon,\epsilon)$.

\begin {lemma}
\label{lemma:induction}
Fix $ a\in\mathcal H ^ g\setminus \mathcal S^ g $ and $\alpha\in S ^ 1 $ with $ df_\alpha $ nonsingular. There exists $\epsilon >0 $ such that for $ t\in (-\epsilon, 0)\cup (0,\epsilon) $, setting
\[
a_t (\lambda) = (\lambda -\alpha e ^ { t}) (\bar\alpha\lambda - e ^ {-  t}) a (\lambda)\in\mathcal H ^ {g +1},
\] 
the degree of $ f_t $ satisfies
\[
\deg (f_t) =\deg (f) +1
\]
and moreover in addition to roots nearby those of $ df $, the differential $ df_t $ has two additional roots on $ S ^ 1 $ in a neighbourhood of $\lambda =\alpha $.
\end{lemma}
\begin {proof}
The differential $df_{\alpha}$ is nonsingular and hence choosing $ f (\alpha) =0 $ we see that there exist $ b_1, b_2\in B_a $ such that $ b_1 $ has a simple root at $\lambda = \alpha $ whilst $ b_2 $ is non-vanishing there. We may take $\delta >0 $ such that for $ t\neq 0 $, the curves $ X_{a_t} $ are nonsingular for $\lambda$ in the open ball $B(\alpha,\delta)\subseteq \C$ at $\alpha$ of radius $\delta$. The curve $ X_{a_0} $ has an ordinary double point at $\lambda =\alpha $.
For  small $ t\geq 0 $, let $ b_{1t}, b_{2t} $ be determined as in Lemma~\ref{lemma:polynomialperturbation}. Since the integrals of $ \Theta_{b_{1t}},\Theta_{b_{2t}} $ over $ A_{g +1} $ are zero, we may find holomorphic functions $ q_{1t} $ and $ q_{2t} $ on each $\lambda^ {- 1} (B(\alpha,\delta))\subset X_{a_t} $ such that
\[
dq_{1t} =\frac {b_{1t} (\lambda)} {y_t}\frac {d\lambda} {\lambda}
\quad\text {and }\quad
dq_{2 t} =\frac {b_{2 t} (\lambda)} {y_t}\frac {d\lambda} {\lambda}.
\]
The above condition determines each $ b_{jt} $ up to an additive constant. We may remove this ambiguity  
by assuming  that for $ t\neq 0 $, $ q_{jt} $ vanishes at the roots $
\alpha e ^ {t} $ and $\alpha e ^ {- t} $ of $ a_t = 0 $ and that $
q_{j 0} (\alpha) = 0 $. Furthermore besides these first order roots
$q_{2 t}$ has no other roots on $\lambda^ {- 1} (B(\alpha,\delta))$.
Then  $\frac {q_{1t}} {q_{2t}} $ is a holomorphic function of $\lambda\in B_{\delta} (\alpha)\subset\C $ and a real-analytic function of $ t\in (-\epsilon,\epsilon) $.
Setting $ t = 0 $, by l'H\^opital's rule the function $\frac {q_{10} (\lambda)} {q_{20}(\lambda)} $ has a simple zero at $\lambda =\alpha $, and so we can take $\epsilon >0 $ such that for each $ t\in (-\epsilon,\epsilon) $, the differential $ d\left (\frac {q_{1t}} {q_{2t}}\right) $ is non-vanishing for $\lambda\in B(\alpha,\delta) $.
As
\[
d \left(\frac {q_{1t}} {q_{2t}}\right) = q_{2t}^ {- 1} d q_{1t} - q_{1t} q_{2t} ^ {-2} d q_{2t},
\]
this implies that for $ t\in (-\epsilon, 0)\cup (0,\epsilon) $ the functions $ b_{1t} $ and $ b_{2t} $ have no common zeros on $ B(\alpha,\delta) $. Thus $\deg (f_t) =\deg(f) +1 $. By the Riemann-Hurwitz formula this is equivalent to the statement that $ df_t $ has two more roots than $ df $, and since both $ f $ and $ f_t $ are real, these additional roots are either interchanged or fixed by $\lambda\mapsto\bar{\lambda} ^ {- 1} $, according to whether or not they lie on the unit circle. Suppose that these roots do not reside on the unit circle. Since $ f ^ {- 1} (\R) $ is preserved by the involution $\lambda\mapsto\bar{\lambda}^{-1}$, so is the complement $ f_t^ {- 1} (\R)\setminus f^ {- 1} (\R) $. These fixed points of $\lambda\mapsto\bar{\lambda}^{-1}$ contained in the additional sheet of $\lambda\mapsto f_{t}(\lambda)$ are disconnected from those contained  in the remaining sheets, since there are no new branch points on the unit circle to connect  the two. But this fixed point set is the unit circle, which is connected and we conclude that the additional roots of $ df_t $ indeed lie on $\{\lambda\in\C\mid \left|\lambda\right| = 1\}$.
\end{proof}


Lemma ~\ref{lemma:induction} shows that if $\mathcal H ^ g\setminus\mathcal S ^ g\neq \emptyset $, then also $\mathcal H ^ {g +1}\setminus\mathcal S ^ {g +1}\neq \emptyset $.
 As noted before the $ g = 0 $ case is trivial so for  each $ g\in\Z^+\cup\{0\} $, the complement of the real subvariety $\mathcal S ^ g $ is an open dense subset of $\mathcal H ^ g $. We now proceed to prove the same of $\mathcal R ^ g $.

\begin {lemma}\label {lemma:dense}
If $\mathcal R ^ g $ is non-empty then it is an open and dense subset of $\mathcal H ^ g $.
\end {lemma}
\begin {proof}
We have shown in Lemmata~\ref {lemma:polynomialperturbation} and \ref{lemma:induction} that $\mathcal S ^ g $ is a real subvariety of $\mathcal H ^ g $ of co-dimension at least one. The complement of $\mathcal R ^ g $ in $\mathcal H ^ g\setminus\mathcal S ^ g $ is again a real subvariety, and so it suffices to show that if $\mathcal R ^ g $ is nonempty then so is its intersection with every component of $\mathcal H ^ g\setminus\mathcal S ^ g $.

Take $ a_0\in\mathcal H ^ g $.  As observed on page~\pageref {Hg}, each $ a\in \mathcal H ^ {g} $ corresponds to $ g $ pairwise distinct elements $\eta_1,\ldots,\eta_g$ of $\{z\in\C\mid 0 <\left| z\right| <1\}$ and so we can complexify $\mathcal H ^ g $ simply by taking pairwise distinct $\eta_1, \ldots ,\eta_{2g}\in\C ^\times $. In a sufficiently small neighbourhood $ U $ of $ a_0 $ in this complexification, we may choose holomorphically varying homology cycles $ A_1, \ldots , A_g $ and for each $a\in U$ define $\mathcal B_a $ to be the space of polynomials $ b $ of degree $ g +1 $ such that $\int_{A_j}\frac {b (\lambda)} {y}\frac {d\lambda} {\lambda} = 0 $ for $ j = 1,\dots, g $. This gives a natural complexification of $\mathcal S ^ g $ and of $\mathcal R ^ g $ in $ U $. Now writing $\mathcal H ^ g_{U ^\C} $ and so on for these local complexifications, the set $\mathcal H ^ g_{U ^\C}\setminus\mathcal S ^ g_{U ^\C} $ is connected. Thus locally we realise $ (\mathcal H ^ g\setminus\mathcal S ^ g)\setminus\mathcal R ^ g $ as the real points of a connected complex subvariety and hence its dimension is locally constant.

Let $ V $ be any component of $\mathcal H ^ g\setminus\mathcal S ^ g $ and let $ V_0 $ be a component such that $ V_0\cap\mathcal R ^ g\neq\emptyset $. Take a continuous  path $\gamma: [-1, 1]\rightarrow\mathcal H ^ g $ joining these two components, and then since the image of $\gamma $ is compact the co-dimension of $\mathcal R ^ g $ in  $\mathcal H ^ g\setminus \mathcal S ^ g $ is constant along $\gamma $ and hence $ V\cap\mathcal R ^ g\neq\emptyset $. 
\end {proof}

To prove that $\mathcal R ^ g $ is indeed non-empty for each $ g $, we first seek a better understanding of the M\"obius invariant quantity $\sum_{\text {roots of } b_1}\Res\,\frac {a} {b_1b_2} d\lambda $, whose non-vanishing characterises the set $\mathcal R ^ g\subset\mathcal H ^ g\setminus\mathcal S ^ g $.

\begin {lemma}\label{lemma:formula}
Take $ a\in\mathcal H ^ g\setminus\mathcal S ^ g $, a basis $ b_1, b_2 $ of $\mathcal B_a $ and as before set $ f =\frac {b_1} {b_2} $. Then
\begin {enumerate}
\item[(a)]
\begin{equation}\label{eq:zeroset}
\sum_{\text {\rm roots of } b_1}\Res\,\frac {a} {b_1b_2} d\lambda\;\; = \sum_{f^ {- 1} (\{0\})}\frac {a} {b_1' b_2 - b_2' b_1}.
\end{equation}
 
\item[(b)] For all M\"obius transformations 
$\displaystyle {\left ( \begin {array} {cc}
A & B\\
C & D\end {array}\right)\in SL (2,\R)} $, the action 
\[
 \tilde b_1 = Ab_1+ Bb_2,\,\tilde b_2 = Cb_1+ Db_2
\]
leaves invariant the function
\[
\tilde b_1'\tilde b_2 -\tilde b_2'\tilde b_1 = b_1' b_2 - b_2' b_1.
\]
\item[(c)] In particular then the following sum is independent of $ p\in\C\P ^ 1 $

\[
\sum_{f^ {- 1} (\{p\})}\frac {a} {b_1' b_2 - b_2' b_1}.
\]
\end {enumerate}

\end {lemma}
\begin {proof}
\begin {enumerate}
\item[(a)] Since $ b_1, b_2 $ are assumed to have no roots in common,
\begin {align*}
\sum_{\text {roots of } b_1}\Res\,\frac {a} {b_1b_2} d\lambda \quad& = \sum_{\text {roots of } b_1}\frac {a} {b_1' b_2}\\
& =\sum_{\text {roots of } b_1}\frac {a} {b_1' b_2 - b_2' b_1}\\
& =\sum_{f^ {- 1} (\{0\})}\frac {a} {b_1' b_2 - b_2' b_1}.
\end {align*}
\item[(b)] This statement is verified by a routine computation.
\item[(c)] Fix $ p\in\C\P ^ 1 $ and take a M\"obius transformation $\displaystyle {\left (
\begin {array} {cc}
A & B\\
C & D\end {array}\right)\in SL (2,\R)} $ such that 
\[
 \frac {A p + B} {C p + D} = 0.
\]
Then for 
\[
\tilde f =\frac {\tilde b_1} {\tilde b_2} =\frac {Af + B} {Cf + D} 
\]
we have
\begin {align*}
\sum_{f^ {- 1} (\{p\})}\frac {a} {b_1' b_2 - b_2' b_1}
& =\sum_{\tilde f^ {- 1} (\{0\})}\frac {a} {b_1' b_2 - b_2' b_1}\\
& =\sum_{\tilde f^ {- 1} (\{0\})}\frac {a} {\tilde b_1' \tilde b_2 - \tilde b_2' \tilde b_1}\text { by (b) of this Lemma}\\
& =\sum_{\text {roots of }\tilde  b_1}\Res\,\frac {a} {\tilde b_1\tilde b_2} d\lambda \text { by (a) of this Lemma}\\
& =\sum_{\text {roots of } b_1}\Res\,\frac {a} {b_1b_2} d\lambda \text { by Lemma~\ref {lemma:independent}}\\
& =\sum_{f^ {- 1} (\{0\})}\frac {a} {b_1' b_2 - b_2' b_1}.
\end {align*}
\end {enumerate}
\end {proof}
In order to prove that $\mathcal R ^ g $ is non-empty for each $ g $, we now investigate the behaviour of the quantity $\sum_{\text {roots of } b_1}\Res\,\frac {a} {b_1b_2} d\lambda $ under the process of ``adding a handle", as in Lemma~\ref {lemma:polynomialperturbation}.
\begin {lemma}\label{lemma:induction 2}
Let $ a\in\mathcal H ^ g\setminus\mathcal S ^ g $ and $\alpha\in S ^ 1 $ be such that $ df_\alpha $ is nonsingular, and take a basis $ b_1, b_2 $ of $\mathcal B_a $ with $ b_1 (\alpha) = 0 $. Fixing a choice of $\sqrt {\bar\alpha} $, let the families $ a_t\in\mathcal H ^ {g +1} $ and $ b_{1t}, b_{2t}\in\mathcal B_{a_t} $ be as defined in Lemma~\ref{lemma:polynomialperturbation}. Then for sufficiently small $ |t|\not=0 $, we have that $ a_t\in\mathcal H ^ {g +1}\setminus\mathcal S ^ {g +1} $ and
\[
\lim_{t\rightarrow 0} \sum_{\text {\rm roots of } b_{1 t}}\Res\,\frac {a_t} {b_{1 t} b_{2 t}} d\lambda\quad =\sum_{\text{\rm roots of }  b_1}\Res\,\frac {a} {b_1b_2} d\lambda -2\;\underset{\lambda =\alpha}{\Res}\;\frac {a} {b_1b_2} d\lambda.
\]
\end {lemma}
\begin {proof} We shall calculate the difference using
  Lemma~\ref{lemma:formula}
\begin{multline}
\sum_{\text {roots of } b_{1 t}}\Res\,\frac {a_t} {b_{1 t} b_{2 t}}d\lambda\quad - \sum_{\text {roots of } b_1}\Res\,\frac {a} {b_1b_2}d\lambda =\\
\begin{aligned}
&=\sum_{f_t^ {- 1} (\{0\})}\frac {a_t} {b_{1t}' b_{2t} - b_{2t}' b_{1t}}\;\; - \sum_{f^ {- 1} (\{0\})}\frac {a} {b_1' b_2 - b_2' b_1}\\
&=\sum_{f_t^ {- 1} (\{\infty\})}\frac {a_t} {b_{1t}' b_{2t} - b_{2t}'b_{1t}}\;\; - \sum_{f^ {- 1} (\{\infty\})}\frac {a} {b_1' b_2 - b_2'b_1}.\label {eq:summation}
\end{aligned}
\end{multline}

Take pairwise disjoint neighbourhoods $ U_{\beta ^ j} $ of each of the roots $\beta ^ j $  of $ b_2 $ and a neighbourhood $ U_\alpha $ of $\lambda =\alpha $ such that each $U_\alpha\cap U_{\beta^j}=\emptyset$. Then by Lemma~\ref{lemma:polynomialperturbation}, for $ \left|t\right| $ sufficiently small, $ U_\alpha $ contains exactly one root of $ b_{2t} $, and $ U_{\beta ^ j} $ contains exactly as many roots of $ b_{2t} $ (counted with multiplicity) as the multiplicity of the root $\beta ^ j $. We summarise this by saying that
in the limit as $ t\rightarrow 0 $, the polynomial $ b_{2t} $ has $ g
+1 $ roots near the roots of $ b_2 $ and one root near $\lambda
=\alpha $. Hence using \eqref {eq:summation} we have
\[
\lim_{t\rightarrow 0}\left (\sum_{\text {roots of } b_{1t}}\Res\,\frac
    {a_t} {b_{1t}b_{2t}} d\lambda\; -  \sum_{\text {roots of }
    b_1}\Res\,\frac {a} {b_1 b_2} d\lambda\right) = 
\lim_{t\rightarrow 0}\frac {-a_t (\beta_t)} {b_{2t}' (\beta_t) b_{1t} (\beta_t)},
\]
where $\beta_t $ denotes the root of $ b_{2 t} $ near $\alpha $.

In the limit as $ t\rightarrow 0 $, the polynomial $ a_t $ has two roots near $\lambda =\alpha $ and $ 2g $ roots near the roots of $ a $, the polynomial $ b_{1t} $ has two roots near $\lambda =\alpha $ and $ g $ roots near the roots of $ b_1 $ on $\C ^\times\setminus\{\alpha\} $, whilst $ b_{2t} $ has one root nearby $\lambda =\alpha $ and $ g +1 $ near the roots of $ b_1 $. Write
\[
a_t (\lambda) =\hat a_t (\lambda)\check a_t (\lambda),
\quad b_{1t} (\lambda) =\hat b_{1t} (\lambda)\check b_{1t} (\lambda),\quad b_{1t} (\lambda) =\hat b_{1t} (\lambda)\check b_{1t} (\lambda),
\]
where $\hat a_t,\hat b_{1t} $ and $\hat b_{2t} $ are polynomials with roots nearby $\lambda =\alpha $ and highest coefficient equal to 1 and $\check a_t,\check b_{1t} $ and $\check b_{2t} $ are polynomials with no roots in a neighbourhood of $\lambda =\alpha $. Using that $\beta_t $ is a root of $\hat b_{2t} $ we have
\[
\frac {-a_t (\beta_t)} {b_{2t}' (\beta_t) b_{1t} (\beta_t)} =\frac {-\hat a_t (\beta_t)} {\hat b_{2t}' (\beta_t)\hat b_{1t} (\beta_t)}\cdot\frac {\check a_t (\beta_t)} {\check b_{2t} (\beta_t)\check b_{1t} (\beta_t)}.
\]
Now $\lim_{t\rightarrow 0}\beta_t =\alpha $ so by the definition of $\check a_t ,\check b_{1t} $ and $\check b_{2t} $ we see that
\[
\lim\limits_{t\rightarrow 0} \frac{\check a_t (\beta_t)} {\check b_{2t} (\beta_t)\check b_{1t} (\beta_t)} =\frac {a (\alpha)} {b_2 (\alpha) b_1' (\alpha)}.
\]
It remains then to calculate the limit
\[
\lim_{t\rightarrow 0}\frac {-\hat a_t (\beta_t)} {\hat b'_{2t} (\beta_t)\hat b_{1t} (\beta_t)}.
\]
In order to calculate this limit we introduce the variable
\[
\kappa=\frac {\lambda -\alpha\cosh t} {i\alpha\sinh t},
\]
which corresponds to blowing up the curve $ X_{a_0} $ at the singularity $\lambda =\alpha $. 

 Note that writing $\kappa $ as a function of $\lambda $ we have $\kappa(\alpha e ^ t) = i $ whilst $\kappa(\alpha e ^ {- t}) = - i $ and the expression for $\lambda $ in terms of $\kappa $ is 
\[
\lambda (\kappa) = i\alpha\sinh (t) \kappa +\alpha \cosh (t).
\]
For $ t>0 $ sufficiently small we have
\begin{equation}\label{eq:factors}
dq_{1t}= \frac{\check b_{1t}(\lambda)}{\lambda\sqrt{\lambda\check a_{t}}} \cdot \frac {\hat b_{1t}(\lambda)d\lambda}{\sqrt{\hat a_t}}\quad\text{and}\quad dq_{2t}= \frac {\check b_{2t}(\lambda)}{\lambda\sqrt {\lambda\check a_{t}}}\cdot\frac{\hat b_{2t}(\lambda)d\lambda}{\sqrt{\hat a_t}}
\end{equation}
where in each expression the first rational function has no poles near $\lambda =\alpha  $ whilst the remaining factor of the differential has no poles near the roots of $ a $. Up to a constant not depending on $t$ the second factors are equal to
\[
 \frac {\hat b_{1t} (\lambda(\kappa))} {\sqrt {\kappa ^ 2+1}} d\kappa\quad\text{and}\quad
\frac {\hat b_{2t} (\lambda(\kappa))} {\sqrt {\kappa ^ 2+1}} d\kappa.
\]
In the limit $t\rightarrow 0$ the first factors in \eqref{eq:factors} converge on all compact subsets of $\kappa\in\C$ uniformly to nonzero constants, since the parameter $\lambda$ converges on these subsets uniformly to $\alpha$. The integral of $ dq_{1t}, dq_{2t} $ along the additional $ A $-cycle around $\kappa = \pm i $ vanishes. As $t\rightarrow 0$, we may choose these additional cycles to be represented by $\kappa\in\partial B(0,2)$. Therefore for all polynomials $p(\kappa)$, which are $t\rightarrow 0$ limits of $t$-dependent linear combinations of $\hat b_{1t} (\lambda(\kappa))$ and $\hat b_{2t}(\lambda(\kappa))$, the form $\frac {p (\kappa)} {\sqrt {\kappa ^ 2+1}} d\kappa$ has no residue at $\kappa=\infty$. The folllowing differentials have this property:
\[
d (\sqrt {\kappa ^ 2+1}) =\frac {\kappa} {\sqrt  {\kappa ^ 2+1}} d\kappa\quad\text {and}\quad d (\kappa\sqrt {\kappa ^ 2+1}) =\frac {2\kappa ^ 2+1} {\sqrt{ \kappa ^ 2+1}} d\kappa.
\]
Moreover $\kappa $ and $\kappa ^ 2+\frac 12 $ span the vector space of all polynomials $ p(\kappa) $ of degree not larger than two such that $\frac {p   (\kappa)} {\sqrt {\kappa ^ 2+1}} d\kappa$ has no residue at $\kappa=\infty$. Since $\hat b_{2t}$ has degree one and $\hat b_{1t}$ degree two and both have highest coefficient one, the renormalised families have the limits
\begin {equation}\label {eq:t0}
\hspace{-2mm}\lim_{t\rightarrow 0}\frac{\hat b_{2t}(\lambda(\kappa))}{i\alpha\sinh(t)}=\kappa\quad\text{and}\quad\lim_{t\rightarrow 0}\frac{\hat b_{1t}(\lambda(\kappa)) -\hat b_{1t}' (\lambda (0))\hat b_{2t} (\lambda(\kappa))}{(i\alpha\sinh(t))^2}=\kappa^2+\frac 12.
\end{equation}
Here the numerator in the second expression is calculated by observing that it is the unique linear combination of $\hat b_{1t} $ and $\hat b_{2t} $ which has highest coefficient one and whose derivative at $\kappa = 0 $ vanishes. The function under consideration is a rational function of degree two: 
\[
\kappa\mapsto\frac {\hat a_t (\lambda(\kappa))}{\hat b'_{1t}(\lambda(\kappa))\hat b_{2t}(\lambda(\kappa))-\hat b'_{2t}(\lambda(\kappa))\hat b_{1t}(\lambda(\kappa))}.
\]
Since the numerator and the denominator have the same degree and the same highest coefficient we can normalise our three polynomials to each have highest coefficient one without changing the rational function. Moreover this function is not changed by adding to $\hat b_{1t}$ a multiple of $\hat b_{2t}$. Hence we have
\[
\lim_{t\rightarrow 0 }
\frac {\hat a_t (\lambda(\kappa))}{\hat b'_{1t}(\lambda(\kappa))\hat b_{2t}(\lambda(\kappa))-\hat b'_{2t}(\lambda(\kappa))\hat b_{1t}(\lambda(\kappa))}=\frac {\kappa ^ 2+1} {2\kappa\cdot \kappa - (\kappa ^ 2+\frac 12)}.
\]
Now since $\beta_t $ is a root of $\hat b_{2t} $ from \eqref {eq:t0} we see that
\[
\lim_{t\rightarrow 0}
\kappa (\beta_t) =\lim_{t\rightarrow 0}
\frac {\beta_t -\alpha\cosh t} {i\alpha\sinh t} = 0,
\]
so
\[
\lim_{t\rightarrow 0 } \frac {\hat a (\beta_t)} {\hat b_{1t}' (\beta_t)\hat b_{2t} (\beta_t) -\hat b_{2t}' (\beta_t)\hat b_{1t} (\beta_t)} =\left.\frac {\kappa ^ 2+1} {\kappa ^ 2 -\frac 12}\right|_{\kappa = 0} = -2.
\]
Thus
\[
\lim_{t\rightarrow 0}\frac {-a_t (\beta_t)} {b_{2t}' (\beta_t) b_{1t} (\beta_t)} = -2\frac {a (\alpha )} {b_1' (\alpha ) b_2 (\alpha )},
\]
completing the proof of the lemma.
\end {proof}
\begin{lemma}\label{lemma:non empty}
For all $g\in\Z^+\cup\{0\}$ the subset $\mathcal R^g\subset\mathcal H^g$ is non-empty.
\end{lemma}

\begin{proof}
For $g=0$ the polynomial $a(\lambda)=1$ and $f$ has degree one so there are no roots of $df$. Therefore the function \eqref{eq:zeroset} does not vanish, and $\mathcal R^0$ is non-empty.

If for the unique $a\in\mathcal H^0$ we add to the corresponding spectral curve $X_a$ a small handle at $\alpha\in S^1$ as in Lemma~\ref{lemma:polynomialperturbation}, then due to Lemma~\ref{lemma:induction 2} in the limit $t\rightarrow 0$ the function \eqref{eq:zeroset} is multiplied by $-1$. This implies that $\mathcal R^1$ is non-empty.

It was shown in Lemma~\ref{lemma:polynomialperturbation} that for each $g\ge 1$ there exists $a\in\mathcal H^g$, such that the corresponding $df$ has roots on $\lambda\in S^1$. Then for such $ a $ the function $\frac {a} {b_1' b_2 - b_2' b_1}$ appearing in \eqref{eq:zeroset}, is not constant on $S^1$. For any $\alpha\in S^1$ we attach a small handle to $ X_a $ at $\alpha $ as in Lemma~\ref{lemma:polynomialperturbation}. By Lemma~\ref{lemma:induction 2}, for almost all $\alpha\in S^1$ in the limit $t\rightarrow 0$ the function \eqref{eq:zeroset} does not vanish. Therefore $\mathcal R^{g+1}$ is non-empty.
\end {proof} 

\bibliographystyle{alpha}



\end {document}